\documentclass{article}
\usepackage[utf8]{inputenc}
\usepackage[T1]{fontenc}
\usepackage{authblk}
\usepackage{xcolor}
\usepackage[
backend=biber,
style=alphabetic,
sorting=ynt
]{biblatex}
\newcounter{mercuriale}

\usepackage[margin=1in]{geometry}  
\usepackage{amsmath}               
\usepackage{amsfonts}              
\usepackage{amsthm}      
\usepackage{amssymb}
\usepackage{bbm}

\usepackage[utf8]{inputenc}
\usepackage{mathtools}

\usepackage{color}   
\usepackage{hyperref}
\hypersetup{
    colorlinks=true, 
    linktoc=all,     
    linkcolor=black,  
}

\def\acts{\curvearrowright}
\DeclarePairedDelimiter\abs{\lvert}{\rvert}%
\DeclarePairedDelimiter\norm{\lVert}{\rVert}%

\makeatletter
\let\oldabs\abs
\def\abs{\@ifstar{\oldabs}{\oldabs*}}
\let\oldnorm\norm
\def\norm{\@ifstar{\oldnorm}{\oldnorm*}}

\makeatother

\theoremstyle{definition}
\newtheorem{thm}{Theorem}[section]
\newtheorem{lem}{Lemma}[section]
\newtheorem{prop}{Proposition}[section]

\newtheorem{cor}{Corollary}[section]

\newtheorem{defn}{Definition}[section]
\newtheorem{example}{Example}[section]

\newtheorem{remark}{Remark}[section]

\DeclareMathOperator{\cost}{cost}

\DeclareMathOperator{\PSL}{PSL}
\DeclareMathOperator{\SL}{SL}

\DeclareMathOperator{\Isom}{Isom}
\DeclareMathOperator{\stab}{stab}

\let\phi\varphi

\let\empt\varnothing

\newcommand{\EE}{\mathbb{E}}      

\newcommand{\RR}{\mathbb{R}}      

\newcommand{\e}{\varepsilon} 
\newcommand{\ZZ}{\mathbb{Z}}      
\newcommand{\QQ}{\mathbb{Q}}      
\newcommand{\PP}{\mathbb{P}}      
\newcommand{\MM}{\mathbb{M}}      

\newcommand{\NN}{\mathbb{N}}      
\newcommand{\HH}{\mathbb{H}}      

\newcommand{\1}{\mathbbm{1}}

\newcommand\restr[2]{{
  \left.\kern-\nulldelimiterspace 
  #1 
  \vphantom{\big|} 
  \right|_{#2} 
  }}

\usepackage{color}
\usepackage{mathrsfs}
\newcommand{\Rel}{\mathcal{R}} 
\newcommand{\MMo}{{\mathbb{M}_0}}

\DeclareMathOperator{\intensity}{int}

\title{Gaboriau's criterion and fixed price one for locally compact groups}
\author{Sam Mellick\footnote{Email: \texttt{samuel.mellick@mcgill.ca}}}
\affil{Department of Mathematics and Statistics, McGill University}

\date{\today}

\begin{document}

\maketitle

\begin{abstract}
Let $G_1$ be a semisimple real Lie group and $G_2$ another locally compact second countable unimodular group. We prove that $G_1 \times G_2$ has fixed price one if $G_1$ has higher rank, or if $G_1$ has rank one and $G_2$ is a $p$-adic split reductive group of rank at least one. As an application we resolve a question of Gaboriau showing $\SL(2,\QQ)$ has fixed price one. Inspired by the very recent work \cite{FMW}, we employ the method developed by the author and Miklós Abért to show that all essentially free probability measure preserving actions of groups weakly factor onto the Cox process driven by their amenable subgroups. We then show that if an amenable subgroup can be found satisfying a double recurrence property then the Cox process driven by it has cost one.
\end{abstract}

\section{Introduction}

\emph{Cost} is a fundamental invariant in measured group theory, first introduced by Levitt\cite{Levitt} and then greatly expanded upon by Gaboriau\cite{GaboriauMerc}\cite{GaboriauCout}. It is a numerical invariant associated to probability measure preserving (pmp) actions of a countable group $\Gamma$ on a standard Borel probability space, and can be viewed as the ergodic theoretic generalisation of the notion of \emph{rank} $d(\Gamma)$, that is, the minimum size of a generating set. A group is said to have fixed price if all of its essentially free pmp actions have the same cost. Aside from its intrinsic interest, cost has numerous applications, for example classifying free groups up to orbit equivalence\cite{GaboriauCout}. Further applications are found in ergodic theory, asymptotic group theory, $3$-manifold topology, percolation theory, the study of paradoxical decompositions, and operator algebra.\cite{GaboriauSeward}\cite{GaboriauLyons}\cite{AbertNikolov}\cite{RussPF}\cite{GaboriauPerc}\cite{EpsteinMonod}\cite{Tarski}\cite{OA1}\cite{OA2}.

Suppose $\Lambda < \Gamma$ is a subgroup of a countable group $\Gamma$. A special case of an argument of Gaboriau (see Critères VI.24 of \cite{GaboriauCout}) states that if $\Lambda$ is \emph{weakly normal} in the sense that $\Lambda \cap \Lambda^\gamma$ is infinite for all $\gamma \in \Gamma$, then
\[
	\cost(\Gamma \acts (X, \mu)) \leq \cost(\Lambda \acts (X, \mu))
\]
for all essentially free pmp actions $\Gamma \acts (X, \mu)$. In particular, if $\Gamma$ contains a weakly normal subgroup of fixed price one, then $\Gamma$ itself has fixed price one. This was used by Gaboriau in the same paper to show that nonuniform lattices in higher rank real semisimple Lie groups have fixed price one. That fact through an orbit equivalence argument also implies that all lattices in such groups admit \emph{some} action of cost one, raising the question of whether or not they too have fixed price one. This was very recently resolved in \cite{FMW}, which serves as the inspiration for the present work.

This paper is fundamentally about an adaptation of Gaboriau's criterion to the realm of locally compact second countable unimodular (lcscu) groups, where substantial new difficulties appear. First, recall the following definition:

\begin{defn}
Suppose $G \acts (X, m)$ is a measure preserving action of a standard Borel space $X$. We say that the action is \emph{conservative} if every positive measure subset $A \subseteq X$ satisfies the conclusion of the Poincaré recurrence theorem. That is, almost every $a \in A$ \emph{recurs}: the set	$\{g \in G \mid ga \in A \}$ is unbounded.
\end{defn}

Observe that for countable groups $\Lambda < \Gamma$, the weak normality condition appearing in Gaboriau's criterion is exactly equivalent with saying that $\Gamma \acts \Gamma/\Lambda \times \Gamma/\Lambda$ is conservative.

We show:

\begin{thm}\label{costtheorem}
	Suppose $A < G$ is an amenable, closed, unimodular, noncompact subgroup such that the action $G \acts (G/A)^2$ is a conservative action. Then $G$ has fixed price one.
\end{thm}

We remark that the above theorem is only of interest if $G/A$ has infinite volume, as otherwise $G$ itself is amenable.

Note that the amenability condition above is required to prove \emph{fixed price} one. Without it, at present we can only conclude that \emph{some} action of $G$ has cost one.

As an application of Theorem \ref{costtheorem} we have:

\begin{thm}\label{higherrankfp}
	Groups of the form $G_1 \times G_2$ with $G_1$ and $G_2$ noncompact lcscu have fixed price one, if:
    \begin{itemize}
        \item $G_1$ is a higher rank real semisimple Lie group, or
        \item $G_1$ is a rank one real simple Lie group and $G_2$ is a product of a $p$-adic split reductive group of rank at least one and an arbitrary (possibly trivial) lcscu group.
    \end{itemize}
\end{thm}

We also apply the above theorem and an argument of Gaboriau(Corollary 2.55 of \cite{GaboriauNotes}) to answer a question of Gaboriau:

\begin{thm}\label{sl2qfp}
	The group $\SL(2,\QQ)$ has fixed price one.
\end{thm}

It was previously known that $\SL(2,\QQ)$ has cost one actions, moreover that the supremum of the cost of its actions is at most $13/12$.

The proof of Theorem \ref{higherrankfp} is essentially a blend of a dynamical variation of Gaboriau's weak normality argument with the techniques developed in \cite{SamMiklos} for handling cost of nondiscrete groups. The dynamical element is heavily inspired by the recent work \cite{FMW}, and recovers some of its results. 

The definition of cost for lcscu groups first appears explicitly in \cite{Carderi}, \cite{SamMiklos}, and \cite{CGMTD}, although it already appears all but in name in Proposition 4.3 of \cite{KPV}, described as folklore.

A fundamental motivation for considering cost for such groups is the following: if one can establish fixed price for an lcscu group $G$, then one proves fixed price for all lattices in $G$ simultaneously. Further motivation is found by \cite{Carderi} and independently \cite{SamMiklos}: for groups of fixed price one, one also gets uniform vanishing of \emph{rank gradient} for so-called Farber sequences of lattices. Recall that if $\Gamma_n < G$ is a sequence of lattices, its rank gradient is
\[
	\lim_{n \to \infty} \frac{d(\Gamma_n) - 1}{\text{vol}(G/\Gamma_n)}.
\]
This was recently applied in \cite{FMW}, resolving a conjecture of Abért, Gelander, and Nikolov \cite{AGN} on vanishing of rank gradient for Farber sequences in higher rank semisimple Lie groups.

In \cite{Carderi} and \cite{CGMTD}, cost is defined for an action $G \acts (X, \mu)$ of a nondiscrete group by looking at its cross-section equivalence relations. In \cite{SamMiklos} an essentially equivalent but stochastic perspective was taken: by first proving that every essentially free pmp action of such a group is isomorphic to some ``point process action'', one can instead just consider the task of defining cost for invariant point processes on $G$. Then cost can be defined as a natural stochastic optimisation problem.

A \emph{point process on $G$} is a random discrete subset $\Pi$ of $G$. It is said to be \emph{invariant} if it is statistically homogeneous in the sense that the shifted point process $g\Pi$ has the same distribution as $\Pi$ for all $g$ in $G$. One then considers all connected \emph{factor graphs} $\mathscr{G}(\Pi)$, which are measurable and equivariant rules for equipping each sample of $\Pi$ with the structure of a connected graph. There is a natural notion of ``average degree'' for factor graphs, and one finally defines the cost of $\Pi$ using the infimum of the average degree of connected factor graphs $\mathscr{G}(\Pi)$. Thus, the cost is the cheapest way to connect up a point process. We will review the relevant definitions in the next section, and one can see \cite{SamMiklos} for a self-contained account.

The list of known results computing cost for nondiscrete groups (a locally compact version of the \emph{mercuriale}, à la Gaboriau) is:

\begin{enumerate}
    \item Groups of the form $G \times \ZZ$ \cite{SamMiklos}, or more generally of the form $G \times \Lambda$ where $\Lambda$ is a countable group containing an infinite order element,
    \item $\Isom(\HH^2)$, $\PSL(2, \RR)$, and $\SL(2, \RR)$\cite{CGMTD},
    \item The automorphism group of a regular tree \cite{CGMTD},
    \item Neretin's group, an unpublished result of the author and Gaboriau, and
    \item Higher rank semisimple real Lie groups, as well as finite products of at least two automorphism groups of regular trees (with possibly varying degrees)\cite{FMW}.
    \setcounter{mercuriale}{\value{enumi}}
\end{enumerate}

Items (2) and (3) above are shown by proving strong treeability results, whilst item (4) uses Gaboriau's criterion and the presence of an open amenable weakly normal subgroup.

Theorem \ref{costtheorem} appends the following groups to the mercuriale:

\begin{enumerate}
\setcounter{enumi}{\value{mercuriale}}
\item Products of noncompact lcscu groups of the form $G_1 \times G_2$, where $G_1$ is a higher rank semisimple real Lie group, and
\item Products of noncompact lcscu groups of the form $G_1 \times G_2$, where $G_1$ is a rank one real Lie simple group and $G_2$ is a $p$-adic split reductive group of rank at least one.
\end{enumerate}

Moreover, any product of such groups with an arbitrary lcscu group also has fixed price one.

Items (1), (5), (6), and (7) have in common that they use a ``weak containment'' type notion which we now discuss.

A point process $\Pi$ \emph{factors} onto another point process $\Upsilon$ if there exists a measurable and equivariantly defined map $\Phi$ such that $\Phi(\Pi)$ has the same distribution as $\Upsilon$. A fundamental fact is that then we have $\cost(\Pi) \leq \cost(\Upsilon)$. The inequality is in this direction as $\Upsilon$ is a ``simpler'' point process than $\Pi$, and thus has fewer connected factor graphs in a sense. Note that this fact also implies cost is an isomorphism invariant. 

Kechris showed \cite{Kechris} that if $\Gamma \acts^\alpha (X, \mu)$ and $\Gamma \acts^\beta (Y, \nu)$ are essentially free pmp actions of a finitely generated group $\Gamma$ and $\beta$ is \emph{weakly contained} in $\alpha$, then $\cost(\alpha) \leq \cost(\beta)$. This was later extended to nonfree actions via groupoid cost in \cite{AbertWeiss}, where it was also shown that the Bernoulli shift $\Gamma \acts [0,1]^\Gamma$ has \emph{maximal} cost. Tucker-Drob extended this to show that every free pmp action has the same cost as its Bernoulli extension \cite{RobinWeakCon}.

In \cite{SamMiklos} a new notion was introduced, modelled on the aforementioned notion of weak containment. One says that $\Pi$ \emph{weakly factors} onto $\Upsilon$ if there exists a sequence of factors $\Phi_n(\Pi)$ which weakly converge to $\Upsilon$. Informally, this means that if one observes $\Phi_n(\Pi)$ in any compact window region of $G$, the resulting process looks more and more like $\Upsilon$ in the statistical sense. The name ``weak factoring'' thus reflects the fact the notion is a weakening of factoring and concerns weak convergence. For compactly generated groups then the cost monotonicity statement $\cost(\Pi) \leq \cost(\Upsilon)$ holds (a weak version of this was shown in \cite{SamMiklos}, the full version can be found in \cite{FMW}).

The following strategy for showing that a group has fixed price one was pioneered in \cite{SamMiklos} and recently applied in \cite{FMW}:
\begin{enumerate}
	\item Identify a point process $\Upsilon$ with cost one, and
	\item Show that every essentially free point process $\Pi$ weakly factors onto $\Upsilon$.
\end{enumerate}

The strategy works due to the monotonicity of cost for weak factoring. To employ it, we introduce the notion of the \emph{factor of IID marking} $[0,1]^\Pi$ of a point process $\Pi$. This process has the same points as $\Pi$, but each point is equipped with a $\texttt{Unif}[0,1]$ random mark. Thus formally speaking it is a ``$[0,1]$-marked point process''. All of the aforementioned results and definitions apply equally well to point processes with marks from any complete and separable metric space. If $\Upsilon$ is a factor of the IID marking $[0,1]^\Pi$, then we say $\Upsilon$ is a \emph{factor of IID factor} of $\Pi$.

It was shown in \cite{SamMiklos} that every essentially free point process $\Pi$ weakly factors onto its own IID marking $[0,1]^\Pi$. This is the analogue of the Abért-Weiss theorem \cite{AbertWeiss}. As weak factoring is a transitive notion (see Theorem 2.24 of \cite{FMW} for a proof), in the above strategy one can use factor of IID factors of $\Pi$ in constructing the weak factor. This is a significantly more tractable problem, first used in \cite{SamMiklos} to prove fixed price for groups of the form $G \times \ZZ$, and recently used in \cite{FMW} to prove fixed price for higher rank semisimple Lie groups and products of automorphism groups of regular trees. 

This strategy is again employed in this paper to establish fixed price for a certain class of groups. 

We now discuss the relation of the present work to its source of inspiration \cite{FMW}, itself inspired by the recent beautiful paper \cite{BCP}. In \cite{BCP} a natural limiting object of the Poisson Voronoi tessellation on $\HH^2$ was identified and used to prove results on the Cheeger constants of large genus surfaces. This model has since been shown to exist for hyperbolic spaces of any dimension \cite{IPVT}, and indeed any Riemannian symmetric space \cite{FMW}. This object, the \emph{ideal Poisson Voronoi tessellation} (IPVT) is an invariant random tessellation of the space into infinite volume cells. It is a canonical object associated to Lie groups and merits further study.

In \cite{FMW} it was identified that the IPVT can be produced as a factor of a \emph{corona action}, which can be identified with the Poisson point process on a certain homogeneous space $G/U$. A highly curious property of the IPVT was identified in higher rank: not only are every pair of cells adjacent, but adjacent along a noncompact region. The IPVT is easily produced as a weak factor of the IID marked Poisson point process, so when constructing factor graphs for the purpose of computing cost one is allowed the randomness of an independent sample of it (see Theorem 2.25 of \cite{FMW}).

To prove fixed price one it therefore suffices to show that the independent coupling of the Poisson point process and the corona action has cost one. The authors achieve this in \cite{FMW} by constructing factor graphs with two types of edges. The first comes from the amenability of the associated subgroup $U < G$, which essentially says that the points of the Poisson inside any particular cell of the IPVT are ``hyperfinite''. The second comes from the aforementioned curious property, and allows one to find an arbitrarily small density of edges crossing between points of the Poisson in different cells. 

In this paper, to prove Theorem \ref{costtheorem} we work exclusively with amenable subgroups $A < G$ satisfying the additional conservativity assumption. We produce a different object as a weak factor of the IID marked Poisson via variation of the ``propagation'' technique introduced in \cite{SamMiklos}. This limiting process is referred to as the Cox process driven by $G/A$. Unlike corona actions (which are point processes on Busemann boundaries of appropriate $G$-spaces), the Cox process is itself a point process on $G$. We use amenability of the subgroup $A$ to construct average degree two ``leafwise hyperfinite'' graphings, and then use conservativity to find arbitrarily sparse sets of edges connecting up the leafwise graphings. This concludes the proof by cost monotonicity. This recovers the real semisimple Lie group case of fixed price one for \cite{FMW}, but not the case of products of trees with varying degrees.

\textbf{Outline of paper}: In Section \ref{Background} we give the necessary definitions on point processes and cost. In Section \ref{weakfactorsection} we execute the first step of the proof strategy and show that the Cox process is a weak factor of IID. In Section \ref{costonesection} we execute the second step and show that this Cox process has cost one, under the additional assumption of conservativity. In Section \ref{applicationssection} we prove Theorem \ref{higherrankfp} by showing they admit appropriate amenable subgroups as well as proving Theorem \ref{sl2qfp}.

For readers unfamiliar with point processes, the paper \cite{SamMiklos} has a quite self-contained exposition of the relevant theory in the notation and language that we will use. A certain familiarity with that paper will be necessary to understand this one. For further resources, a very basic introduction to point processes is \cite{Kingman}. The recent book \cite{LastPenrose} is a helpful resource, as well as the lecture notes \cite{Bacc}. We will also cite the reference works \cite{DVJ1}, \cite{DVJ2}, and \cite{Kallenberg}. 

\textbf{Acknowledgements}: The author wishes to acknowledge Miklós Abért, Miko\l{}aj Fr\k{a}czyk, Antoine Poulin, Anush Tserunyan, and Konrad Wrobel for their encouragement and useful conversations that led to the present work.

\section{Background}\label{Background}

Let $G$ denote a locally compact second countable unimodular (lcscu) group. If $G \acts (X, \mu)$ is an essentially free probability measure preserving action on a standard Borel probability space $(X, \mu)$, then one can define the \emph{cost} of the action using what are called cross-section equivalence relations. An alternative approach to defining cost was developed in \cite{SamMiklos}, by using the fact that all such actions are isomorphic to \emph{point process actions}. As cost is an isomorphism invariant, there is no loss in generality in just considering cost for point process actions. We now recall the key definitions.

\subsection{Point processes on homogeneous spaces}

Let $G$ denote a locally compact second countable (lcsc) group. We will typically assume $G$ is also unimodular, although some of what follows holds without that assumption. Such groups admit a proper\footnote{A metric is \emph{proper} if closed balls with respect to it are compact.} left-invariant metric $d$ which we fix. They also admit a left invariant Haar measure $\lambda$, which we also fix. We will assume that $G$ is noncompact and nondiscrete (that is, $\lambda(G) = \infty$ and $\lambda$ is atomless).

The goal of this article is the following: we show that if $G$ contains an appropriate amenable subgroup, then $G$ has fixed price one, meaning all of its essentially free probability measure preserving (pmp) actions have cost one. We review the necessary definitions. Note that every essentially free pmp action of $G$ is isomorphic to some point process, which we define now.

We will study point processes on homogeneous spaces of $G$. Explicitly, this means the study of point processes on spaces of the form $G/H$ where $H < G$ is a closed subgroup. Throughout the article, $H$ always refers to a closed subgroup of $G$. We will require the existence of an invariant measure on $Q:=G/H$, which in our situation is equivalent with saying $H$ is also unimodular. Note that $G$ itself is a homogeneous space. Recall that the Haar measure $\lambda_H$ on $H$ and $\lambda_Q$ on $G/H$ can be chosen so that the fundamental equation holds:
\begin{equation}\label{haarrel}
	\lambda(B) = \int_{Q} \lambda_{gH}(B) d\lambda_Q(gH),
\end{equation}
where $\lambda_{gA}(B) = \lambda_A(g^{-1}B)$ is the natural Haar measure on the coset $gH$. 

The \emph{configuration space} of $G/H$ is
\[
	\MM(G/H) = \{ \omega \subset G \mid \omega \text{ is locally finite} \}.
\]

Let $\Xi$ denote a complete and separable metric space (csms) of marks. We think of $\Xi$ as labels or colours. Frequently, $\Xi = [0,1]$ or $\Xi = [d]$, where $[d] = \{1, 2, \ldots, d\}$. The \emph{$\Xi$-marked configuration space} of $G/H$ is the space of configurations where each point has a label from $\Xi$ attached. Formally, it is
\[
	\Xi^\MM = \{ \omega \in \MM(G/H \times \Xi) \mid \pi(\omega) \in \MM(G/H) \text{ and if } (g, \xi_1) \in \omega \text{ and } (g, \xi_2) \in \omega \text{ then } \xi_1 = \xi_2 \},
\]
where $\pi : G/H \times \Xi \to G/H$ is the projection. 

\begin{defn}
	A \emph{point process} on $G/H$ is a random element $\Pi$ of $\MM(G/H)$. Formally speaking, this means there is an auxilliary probability space $(\Omega, \PP)$, and $\Pi \colon \Omega \to \MM(G/H)$ is a measurable map. As is standard, we will use expressions like $\PP[\Pi \in A]$ for what is formally $\PP[\{\omega \in \Omega \mid \Pi(\omega) \in A\}]$. All definitions below apply for $\Xi$-marked point processes as well, we only make this explicit if there is some care required.
	
	The \emph{distribution} or \emph{law} of $\Pi$ is then the pushforward $\mu := \Pi_*(\PP)$.
	
	The point process is \emph{invariant} if for all $g \in G$, $g\Pi$ and $\Pi$ have the same distribution. Equivalently, if $G \acts (\MM(G/H), \mu)$ is a probability measure preserving (pmp) action. We say that $\Pi$ is \emph{essentially free} if $\stab_G(\Pi) = 1$ almost surely. That is, if the associated pmp action $G \acts (\MM(G/H), \mu)$ is essentially free.
\end{defn}

\begin{remark}
	If $\Pi$ is a $\Xi$-marked point process, then we will abuse notation and write ``$g \in \Pi$'' to more formally mean that there exists some $\xi_g \in \Xi$ such that $(g, \xi_g) \in \Pi$. In that case, we also write $\xi_g$ for the label of $g \in \Pi$.
\end{remark}

One can define the configuration space for any metric space. If the underlying metric space is complete and separable, then so too is its configuration space (see Theorem A2.6.III of \cite{DVJ1}). We will not use the metric structure explicitly, but we will make use of its measurable structure. 

\begin{defn}
	Let $B \subseteq G$ be Borel. Its associated \emph{point count functional} is the map $\omega \mapsto \omega(B)$, where $\omega(B) = \abs{\omega \cap B}$. The measurable structure on $\MM(G/H)$ is exactly so that all of the point count functionals are measurable.
\end{defn}

To explain the notation, note that one can identify $\omega$ with the counting measure $\sum_{x \in \omega} \delta_x$, and $\omega(B)$ is then the measure of $B$. 

\begin{defn}
	The \emph{intensity} of an invariant point process $\Pi$ on $G/H$ is
	\[
		\intensity(\Pi) = \frac{\EE[\Pi(U)]}{\lambda_Q(U)},
	\]
	where $0 < \lambda_Q(U) < \infty$ is a measurable subset of $G/H$. 
\end{defn}

The intensity is well-defined (in the sense that it doesn't depend on the choice of $U$) by the uniqueness up to scaling of the measure $\lambda_Q$.

\begin{thm}[Campbell theorem]
	Let $\Pi$ be an invariant point process on $G/H$ and $f \colon G/H \to \RR_{\geq 0}$ a measurable function. Then
	\[
		\EE\left[\sum_{x \in \Pi} f(x)\right] = \intensity(\Pi) \int_{G/H} f(x)d\lambda_Q(x).
	\]
\end{thm}

The proof of this is a standard monotone class argument. Note that the theorem with $f(x) = \1[x \in U]$ recovers the definition of intensity. We will make use of the following simple application of the Campbell theorem, which is likely folklore:

\begin{cor}\label{Meckecor}
	Suppose $\Pi$ is an invariant point process on $G/H$, and $B \subseteq G/H$ has zero measure. Then $\PP[\Pi(B)=0] = 1$. That is, almost surely no point of $\Pi$ is in $B$.
\end{cor}

\begin{proof}
	Apply the Campbell theorem with $f(x) = \1[x \not\in B]$. 
\end{proof}

We think of Corollary \ref{Meckecor} in the following way: $B$ is a subset of ``bad'' points. The corollary states that almost surely no point of an invariant point process is bad. Note that the corollary is immediate for the Poisson point process (to be discussed next). This is in fact the case we will use, but we include the above lemma to illustrate the Campbell theorem.

\subsection{Examples of point processes}
\begin{example}
	Suppose $\Gamma < G$ is a lattice, that is, a discrete subgroup such that $G \acts G/\Gamma$ carries an invariant probability measure. Note that all cosets $a\Gamma \in G/\Gamma$ are configurations, and thus the invariant probability measure defines an invariant point process on $G$.
\end{example}

\begin{example}
	A point process $\Pi$ on $G/H$ is said to be \emph{Poisson with intensity $t > 0$} if the following two conditions are satisfied:
	\begin{enumerate}
		\item For all Borel $B \subseteq G/H$, the random variable $\Pi(B)$ is Poisson with parameter $t\lambda_Q(B)$, that is
		\[
			\PP[\Pi(B) = k] = e^{-t\lambda_Q(B)} \frac{(t\lambda_Q(B))^k}{k!}.
		\]
		\item For all Borel sets $B_1, B_2 \subseteq G/H$ which are disjoint, the random variables $\Pi(B_1)$ and $\Pi(B_2)$ are independent.  
	\end{enumerate}
\end{example}

\begin{example}
	Suppose $\Pi$ is an invariant point process. We denote by $[0,1]^\Pi$ its \emph{IID marking}. This is the $[0,1]$-marked point process which is $\Pi$ decorated with independent $\texttt{Unif}[0,1]$ labels.
\end{example}

Another name for the IID marking is \emph{Bernoulli extension}. 

\begin{example}
Suppose $\eta$ is an invariant random measure on $G$. Then the \emph{Cox process} driven by $\eta$ is the Poisson point process with respect to $\eta$.
\end{example}

\begin{defn}

Suppose $H < G$ is a closed and unimodular subgroup. One can view the Poisson point process $\Upsilon$ on $G/H$ as defining an invariant random measure on $G$, namely
\[
	\eta_\Upsilon = \sum_{gH \in \Upsilon} \lambda_{gH}.
\]
We will refer to the Cox process driven by this invariant random measure as \emph{the Cox process driven by $G/H$}.

\end{defn}

\begin{lem}\label{coxintensity}
	The Cox process $\Pi$ driven by $G/H$ has intensity one.
\end{lem}

\begin{proof}
	The intensity of $\Pi$ is simply the intensity of the driving measure (see Equation 13.7 of \cite{LastPenrose}). We therefore compute, using $U \subseteq G$ of unit volume:
	\[
		\intensity(\Pi) = \EE[\eta_\Upsilon(U)] = \EE\left[\sum_{gH \in \Upsilon} \lambda_{gH}(U)\right].
	\]
	Note that the only terms contributing to the sum are those $gH$ in $\Upsilon \cap UH$. There are $\mathtt{Pois}(\lambda_Q(UH))$ many of these, each contributing on average $\lambda(gH)(U)$ where $gH$ is chosen uniformly at random over $UH$. Thus
	\[
	\EE\left[\sum_{gH \in \Upsilon} \lambda_{gH}(U)\right] = \lambda_Q(UH)\frac{1}{\lambda_Q(UH)} \int_{UH} \lambda_{gH}(U)d\lambda_Q(gH) = \lambda_G(U),
	\]
	as desired.
\end{proof}

\begin{defn}
	Suppose $\Pi$ and $\Upsilon$ are invariant point processes. We say that $\Upsilon$ is a \emph{factor} of $\Pi$ if there exists a measurable and equivariant map $\Psi$ such that $\Psi(\Pi)$ has the same distribution as $\Upsilon$. 
\end{defn}

\begin{example}
	Suppose $\Pi$ is an invariant point process. Then its IID marking $[0,1]^\Pi$ factors onto the Poisson point process (of any intensity). For a proof, see for example Proposition 5.1 of \cite{SamMiklos}.
\end{example}

\begin{example}
	Suppose $\Pi$ is an invariant point process, and $\Upsilon$ is a factor of $\Pi$. Then the IID marking $[0,1]^\Pi$ factors onto the IID marking $[0,1]^\Upsilon$.
\end{example}

\subsection{Weak convergence}

\begin{defn}
Let $\Pi$ be a point process. Then the \emph{finite dimensional distributions} (or simply \emph{fidi's}) of $\Pi$ are the totality of random vectors
\[
	(\Pi(B_1), \Pi(B_2), \ldots, \Pi(B_k)) \in \NN_0^k
\]
for all $k \in \NN$ and choices of Borel sets $B_1, \ldots, B_k \subseteq G$. 
\end{defn}

Point processes are determined by their fidi's. In fact, one needs considerably fewer choices than ``all Borel sets''. 

\begin{defn}
	Let $\Pi$ be a point process. Then $B \subseteq G$ is a \emph{stochastic continuity set} for $\Pi$ if $\PP[\Pi \cap \partial B \neq \empt] = 0$, where $\partial B$ denotes the topological boundary of $B$.
\end{defn}

\begin{defn}
	Let $\Pi_n$ be a sequence of point processes, and $\Pi$ another point process. We say $\Pi_n$ \emph{weakly converges} to $\Pi$ if for each $k \in \NN$ and choices $B_1, B_2, \ldots, B_k$ of bounded stochastic continuity sets for $\Pi$, the corresponding fidi's of $\Pi_n$ converge to that of $\Pi$. Explicitly, for each vector $\alpha \in \NN_0^k$, we have
	\[
		\lim_{n \to \infty}\PP[(\Pi_n(B_1), \Pi_n(B_2), \ldots, \Pi_n(B_k) = \alpha] = \PP[(\Pi(B_1), \Pi(B_2), \ldots, \Pi(B_k) = \alpha].
	\]
\end{defn}

\begin{remark}
	As discussed in Proposition 11.1.VIII of \cite{DVJ2}, in order to verify weak convergence of random measures $\eta_n$ it suffices to check the weak convergence $\eta_n(B)$, where $B$ ranges over a \emph{countable} collection of compact sets.
\end{remark}

The following concept is central to the work on cost on nondiscrete groups:

\begin{defn}
	Suppose $\Pi$ and $\Upsilon$ are invariant point processes on $G$. We say that $\Pi$ \emph{weakly factors} onto $\Upsilon$ if there exists a sequence of factors $\Phi_n(\Pi)$ which weakly converge to $\Upsilon$. 
\end{defn}

\begin{example}
	Suppose $\Pi$ is an essentially free point process. Then Theorem 5.5 of \cite{SamMiklos} that $\Pi$ weakly factors onto its IID marking $[0,1]^\Pi$.
\end{example}

\begin{remark}
	It is not straightforward, but one can show that weak factoring is a transitive notion (see Theorem 2.24 of \cite{FMW}).
\end{remark}

\begin{thm}\label{convergencecondition}[See for example Theorem 4.17 of \cite{Kallenberg}]
	Let $\Pi_n$ be a sequence of Cox processes, driven by random measures $\eta_n$. Then $\Pi_n$ weakly converges if and only if and only if $\eta_n$ weakly converges. In that case the weak limit of $\Pi_n$ is a Cox process directed by the limit of $\eta_n$.
\end{thm}

\subsection{Factor graphs}

\begin{defn}
	Let $\Pi$ be an invariant point process. A \emph{factor graph} of $\Pi$ is a measurable and equivariantly defined graph with vertex set $\Pi$. More formally, it is a measurable and equivariant map $\mathscr{G} : \MM(G/H) \to \MM(G/H \times G/H)$ such that $\mathscr{G}(\Pi) \subseteq \Pi \times \Pi$.
\end{defn}

\begin{remark}
	Formally, the above definition defines \emph{directed} factor graphs. One can easily adapt the definition to have undirected graphs, multigraphs, labelled graphs, et cetera. We will follow the convention of probabilists and leave the space implicit.
\end{remark}

\begin{example}
	The \emph{distance-$R$} factor graph of a point process $\Pi$ on $G$ is
	\[
		\mathscr{D}_R(\Pi) = \{ (g, h) \in \Pi \times \Pi \mid d(g, h) < R \}.
	\]
\end{example}

\begin{defn}
	Let $\Pi$ be an invariant point process. A \emph{factor of IID factor graph} of $\Pi$ is any factor graph of the IID marking $[0,1]^\Pi$.
\end{defn}

\begin{example}
	Let $\Pi$ be an invariant point process on $G$. The factor of IID \emph{$t$-cutoff $R$-star factor graph} $\bigstar_{t,R}$ is defined 
	\[
	\bigstar_{t,R}([0,1]^\Pi) = \{(g, h) \in \Pi \mid g \neq h, \xi_g < t \text{ and } d(g, h) < R \}.
	\]
	In words, points with label lower than the threshold $t$ connect to all points in their $R$-neighbourhood. 
\end{example}

This factor graph is essentially a percolated version of the distance-$R$ factor graph.

\subsection{The Palm measure}

Suppose $\Pi$ is a point process on $G$. We denote the identity element by $0 \in G$ and think of it as a ``root'' of the space. We would like to condition on $\Pi$ to contain $0$, but this is a probability zero event. One can condition on $\Pi$ to contain a point in a very small ball about $0$ and then take a limit, but we will instead take the alternative equivalent approach.

\begin{defn}
	The \emph{rooted configuration space} of $G$ is 
	\[
		\MMo(G) = \{ \omega \in \MM(G) \mid 0 \in \omega \}.
	\]
	More generally, the \emph{$\Xi$-marked rooted configuration space} of $G$ is
	\[
		\Xi^{\MMo(G)} = \{ \omega \in \MM(G) \mid \exists \xi \in \Xi \text{ such that } (0, \xi) \in \omega \}.
	\]	
	
	Suppose $\Pi$ is an invariant point process on $G$ of finite, nonzero intensity. Its \emph{Palm measure} is the probability measure $\mu_0$ on $\MMo$ defined for measurable $C \subseteq \MMo(G)$ by 
	\[
		\mu_0(C) = \frac{1}{\intensity(\Pi)} \EE\left[ \#\{g \in U \mid g^{-1}\Pi \in C\} \right]
	\]
	We denote by $\Pi_0$ a random variable with distribution $\mu_0$ and refer to it as a \emph{Palm version} of $\Pi$. We also define these notions for $\Xi$-marked point processes in the same way.
\end{defn}

\begin{example}
	The Palm measure of a lattice shift $G \acts G/\Gamma$ is simply $\delta_\Gamma$, thus $\Gamma$ itself is a Palm version.
\end{example}

\begin{example}
	A Palm version for the Poisson point process $\Pi$ is $\Pi_0 = \Pi \cup \{0\}$. This property characterises the Poisson point process.
\end{example}

\begin{example}
	Suppose $\Pi_0$ is a Palm version. Then $[0,1]^{\Pi_0}$ is a Palm version for the IID marking $[0,1]^\Pi$.
\end{example}

\begin{example}
	Let $\Pi$ be the Cox process driven by $G/H$, and $\Pi_0'$ an independent sample of the Palm version of the Poisson on $H$. Then $\Pi \cup \Pi_0'$ is a Palm version of $\Pi$.
\end{example}

\begin{defn}
	The \emph{rerooting equivalence relation} $\Rel$ on $\MMo(G)$ is the restriction of the orbit equivalence relation of $G \acts \MM(G)$ to $\MMo(G)$. Explicitly, for a rooted configuration $\omega \in \MMo(G)$ we declare
	\[
		\omega \sim g^{-1}\omega \text{ for all } g \in \omega.
	\]
\end{defn}

If $\Pi$ is an invariant point process on a unimodular group $G$, then $(\MMo(G), \Rel, \mu_0)$ is a pmp countable Borel equivalence relation (cber) called the \emph{Palm equivalence relation} of $\Pi$. If $\Pi$ is also essentially free, then there is a correspondence between graphings of its Palm equivalence relation and factor graphs of $\Pi$. This fundamental fact about the Palm equivalence relation and its connection with factor objects, was first observed in \cite{Mellick}. For example, let $\mathscr{G}(\Pi)$ denote a factor graph of $\Pi$, a point process with distribution $\mu$. Then $\mathscr{G}(\Pi)$ has a well-defined restriction to $\MMo(G)$, \emph{despite} the fact that this is a $\mu$ measure zero subset. In fact, this restriction is defined $\mu_0$ almost everywhere, where $\mu_0$ is the Palm measure. Next observe that if $\omega \in \MMo(G)$ has trivial stabiliser, then the map
\[
	p_\omega(g) : \omega \to [\omega]_\Rel \text{ given by } g \mapsto g^{-1}\omega
\]
is \emph{bijective}. Thus we can look at the graph structure on $\mathscr{G}(\Pi_0)$ as defining a graph structure on $[\Pi_0]_\Rel$. The converse holds too. In this way we have seen that there is a one to one correspondence
\[
	\{\text{factor graphs of } \Pi \} \leftrightsquigarrow \{\text{subgraphs of } (\Rel, \widetilde{\mu_0}) \},
\]
where $\widetilde{\mu_0}$ denotes the natural lift of $\mu_0$ to $\Rel$. 

Similar correspondences exist for other factor objects. We freely make such identifications wherever necessary. We now discuss another case of this correspondence principle relevant to the present work.

\begin{defn}
	Let $\Pi$ be an invariant point process. A \emph{factor equivalence relation} $\sim$ on $\Pi$ is a measurably and equivariantly defined equivalence relation the set $\Pi$. To formalise this, one can consider $\sim$ as a factor graph of $\Pi$ whose connected components are complete graphs.
\end{defn}

Under the above mentioned correspondence, factor equivalence relations are exactly the same thing as subequivalence relations of the Palm equivalence relation. 

The following definition was introduced in \cite{HolroydPeres} and later used in \cite{Timar}:

\begin{defn}
	Suppose $\Pi$ is an invariant point process on $G$. A \emph{one-ended clumping} is a sequence of factor equivalence relations $\sim_n(\Pi)$ such that
	\begin{itemize}
		\item All equivalence classes of $\sim_n(\Pi)$ are finite,
		\item It increases: $\sim_n(\Pi) \subseteq \sim_{n+1}(\Pi)$, and
		\item The union of all $\sim_n(\Pi)$ is the trivial equivalence relation (just one class).
	\end{itemize}
\end{defn}

It was observed in \cite{Mellick} that an essentially free point process $\Pi$ on $G$ admits a one-ended clumping if and only if $G$ is amenable. We make use of a generalisation of that fact.

\begin{defn}
	Suppose $(X, \Rel, \mu)$ is a pmp cber. We say it is \emph{amenable} if it admits an invariant mean, in the sense that for each $x \in X$ one has a function $m_x : \ell^\infty([x]_\Rel) \to \RR$ such that $m_x(1) = 1$ and $m_x = m_y$ if $x \Rel y$. It is \emph{$\mu$-amenable} if the functions $m_x$ are only defined $\mu$ almost everywhere.
	
	A pmp cber is \emph{hyperfinite} if it can be expressed as a countable union of equivalence relations with all classes finite. It is \emph{$\mu$-hyperfinite} if this holds $\mu$ almost everywhere.
\end{defn}

It is known that $\mu$-amenability and $\mu$-hyperfiniteness are equivalent. Observe that under the correspondence, a point process $\Pi$ admits a one-ended clumping if and only if the associated Palm equivalence relation is hyperfinite.

\begin{defn}
	Suppose $H < G$ is a closed subgroup. An invariant point process $\Pi$ is said to be \emph{leafwise} if $\Pi_0 \cap H$ is the Palm version of an invariant point process on $H$.
\end{defn}

\begin{example}
	Suppose $H < G$ is a closed and unimodular subgroup. Then the Cox process driven by $G/H$ is leafwise.
\end{example}

\begin{defn}
	Suppose $H < G$ is a closed subgroup. A \emph{leafwise one-ended clumping} is a sequence of factor equivalence relations $\sim_n(\Pi)$ such that
	\begin{itemize}
		\item All equivalence classes of $\sim_n(\Pi)$ are finite,
		\item It increases: $\sim_n(\Pi) \subseteq \sim_{n+1}(\Pi)$, and
		\item The union of all $\sim_n(\Pi)$ is the equivalence relation on $\Pi$ given by $x \sim y$ if and only if $xH = yH$.
	\end{itemize}
\end{defn}

\begin{defn}
	Suppose $H < G$ is a closed subgroup. Let $\Rel\Rel$ denote the \emph{restricted rerooting} subequivalence relation of the rerooting equivalence relation, defined by for $\omega \in \MMo(G)$ by
	\[
		\omega \sim h^{-1}\omega \text{ for all } h \in \omega \cap H.
	\]
\end{defn}

Under the correspondence principle, we have:

\begin{lem}\label{leafwiseclumpiffhyperfinite}
	A leafwise point process admits a leafwise one-ended clumping if and only if the restricted rerooting equivalence relation $(\MMo, \Rel\Rel, \mu_0)$ is $\mu_0$-hyperfinite.
\end{lem}

The following is the aforementioned generalisation of a result in \cite{Mellick} (take $H = G$).

\begin{prop}\label{leafwiseamen}
	Let $H < G$ be a closed and unimodular subgroup. Suppose $\Pi$ is a leafwise point process on $G$. Then $\Pi$ admits a leafwise one-ended clumping if and only if $H$ is amenable. 
\end{prop}

\begin{defn}
	Let $\omega \in \MM(G)$ be a configuration. The \emph{Voronoi cell} of a point $g \in \omega$ is 
	\[
		V_\omega(g) = \{ x \in G \mid d(g, x) \leq d(\gamma, x) \text{ for all } \gamma \in \omega \}.
	\]
	The \emph{Voronoi tessellation} associated to $\omega$ is the ensemble of Voronoi cells $\{V_\omega(g)\}_{g \in \omega}$.
\end{defn}

Note that the Voronoi tessellation is equivariantly defined since the metric is left invariant.

The Vornoi tessellation is really a covering of $G$ by closed sets. In principle the boundaries of these sets can have positive volume, or at least this question is sensitive to the metric. If a true partition is required (which is often the case), one can instead use a tie-breaking mechanism to refine the cells at the expense of making them only measurable instead of closed. To do this, fix any Borel isomorphism $I : \MM(G) \to [0,1]$. Define the $I$-refined Voronoi cell of $\omega \in \MM(G)$ by
\[
	V^I_\omega(g) = \{ x \in G \mid d(g, x) \leq d(\gamma, x) \text{ for all } \gamma \in \omega \text{ and } I(g^{-1}\omega) < I(\gamma^{-1}\omega) \text{ if } d(\gamma,x) = d(g, x)\}.
\]
In other words, each point $x$ chooses the closest point of $\omega$ to it, and if this is not unique then it picks the point with smallest label. For our purposes there is no difference between ``true'' Voronoi cells and these tie-broken ones, so we use them and suppress $I$ from the notation.

An important fact for what follows is that almost surely, the Voronoi cell of every point of an invariant point process has finite volume. This follows from the Voronoi inversion formula, see for instance Section 9.4 of \cite{LastPenrose}. 

\begin{defn}
	Suppose $H < G$ is a closed subgroup and $\omega \in \MM(G)$ is a configuration. The \emph{leafwise Voronoi cell} of $g$ with respect to $\omega$ is
	\[
		LV_\omega(g) = \{ x \in gH \mid d(g, x) \leq d(\gamma, x) \text{ for all } \gamma \in \omega \cap gH\}.
	\]
\end{defn}

Again, one can further refine this with a tie-breaking mechanism, which we do but suppress in notation. Note that for a ``typical'' process the leafwise Voronoi cell is simply $gH$ itself, but for the Cox process driven by $G/H$ it is a nontrivial tessellation of $H$.

\begin{proof}[Proof of Proposition \ref{leafwiseamen}]

 Suppose $H$ is amenable. By Lemma \ref{leafwiseclumpiffhyperfinite}, it suffices to show that the restricted rerooting equivalence relation is $\mu_0$ amenable. Fix a left invariant mean $M : L^\infty(H) \to \RR$. Given a function $f \in \ell^\infty([\omega]_{\Rel\Rel})$, we define its extension $F : H \to \RR$ in the following way. Given a leafwise configuration $\omega$ and $h \in H$, there is a unique element $X(\omega, h)$ of $\omega$ such that the leafwise Voronoi cell of $X(\omega, h)$ contains $h$. Define
	\[
		F(h) = f(X(\omega, h)),
	\]
	and
	\[
		(m_\omega)(f) = M(F).
	\]
	It is clear that this defines an invariant mean.
	
	For the converse, assume that the restricted rerooting equivalence relation is $\mu_0$ hyperfinite and the point process is leafwise. Let us write $\Pi^H$ for the invariant point process on $H$ such that the Palm version $\Pi^H_0$ of $\Pi^H$ is $\Pi_0 \cap H$. We write $\nu$ and $\nu_0$ for the distribution of $\Pi_H$ and its Palm measure respectively.
 
	We construct an invariant mean $M : L^\infty(H) \to \RR$ as follows, using a standard technique from percolation theory. Let $V_n(\Pi_0)$ denote the union of the leafwise Voronoi cells of the points $\sim_n(\Pi^H_0)$ equivalent to the identity. For an essentially bounded function $F : H \to \RR$, let
	\[
		M_n(F) = \frac{1}{\lambda_H(V_n(\Pi^H_0))}\int_{V_n(\Pi^H_0)} F(h)d\lambda_H(h) \text{ and } M(F) = \lim_{n \to \mathcal{U}} M_n(F),
	\]
where $\mathcal{U}$ denotes a nonprincipal ultrafilter on $\NN$. This is invariant: for $n$ sufficiently large (depending randomly on $\Pi^H_0$), $M_n(h\cdot F) = M_n(F)$.
\end{proof}

\subsection{Cost}

\begin{defn}
	Suppose $\mathscr{G}(\Pi)$ is a factor graph. Its \emph{average degree} is
	\[
		\texttt{AvgDeg}(\mathscr{G}(\Pi)) = \left(\EE[\deg_0(\mathscr{G}(\Pi_0))\right],
	\]
	where $\Pi_0$ is the Palm version of $\Pi$.
\end{defn}

The average degree can be expressed without reference to the Palm measure, see Remark 3.14 of \cite{SamMiklos}

\begin{example}
Consider the $t$-cutoff $R$-star factor graph $\bigstar_{t,R}$. By the explicit description of the Palm version one readily sees that its average degree is $t\lambda(B(0,R))$. 
\end{example}

\begin{defn}
	Suppose $\Pi$ is an invariant point process on $G$. Its \emph{cost} is defined by the following formula
	\[
		\cost(\Pi) - 1 = \intensity(\Pi)\inf_{\mathscr{G}} \left(\frac{1}{2}\EE[\deg_0(\mathscr{G}(\Pi_0))] - 1 \right),
	\]
	where the infimum ranges over all connected factor graphs of $\Pi$.
\end{defn}
Thus to prove a group has fixed price one, one must prove that every essentially free point process $\Pi$ admits ``cheap'' connected factor graphs, that is for every $\e > 0$ a connected factor graph with average degree at most $2 + \e$. 

Cost is an isomorphism invariant. The following generalisation of that fact is crucial for us:

\begin{thm}
	Let $\Pi$ be an essentially free point process  weakly factoring onto $\Upsilon$. Then $\cost(\Pi) \leq \cost(\Upsilon)$.
\end{thm}

The above theorem essentially appears in \cite{SamMiklos}, and appears explicitly as Theorem 2.22 of \cite{FMW}.

\begin{cor}\label{fiidcost}
	Suppose $\Pi$ is an essentially free point process on $G$. Then $\cost(\Pi) = \cost([0,1]^\Pi)$.
\end{cor}

One should interpret this corollary as follows: in the definition of cost, one could allow the infimum to range over all factor of IID factor graphs, and the resulting cost notion is the same. 

\section{The IID Poisson weakly factors onto the Cox process}\label{weakfactorsection}

\begin{thm}\label{constructweakfactor}
	Suppose $A < G$ is a closed, amenable, unimodular subgroup. Then the IID Poisson point process on $G$ weakly factors onto the Cox process driven by $G/A$. 
\end{thm}

\begin{remark}
We will assume that the groups $A$ and $G$ are nondiscrete, as this is the more difficult case. The proof carries over to the discrete case by replacing references to the Poisson point process with Bernoulli percolation. Note that our main interest in the theorem is its application to cost, for which the nondiscrete case implies the discrete one: if $A$ is discrete, then $A \times S^1$ is nondiscrete in $G \times S^1$. The latter is then shown to have fixed price one. Note that $G < G \times S^1$ is a closed subgroup of finite covolume, so it inherits fixed price one.
\end{remark}

\begin{cor}\label{coxweakfactor}
	Every essentially free point process on $G$ weakly factors onto the Cox process driven by $G/A$. 
\end{cor}

Theorem \ref{constructweakfactor} can be seen as a natural generalisation of a result used in the proof that groups of the form $G \times \ZZ$ have fixed price one. In \cite{SamMiklos} it is shown that the IID Poisson on $G \times \ZZ$ weakly factors onto the Poisson on $G$ times $\ZZ$, by using the notion of ``propagation''. This consisted of taking a lower intensity Poisson point process on $G \times \ZZ$ and ``propagating'' it vertically. In the proof, it's quite explicit that amenability of $\ZZ$ plays a crucial role in the weak convergence. We will consider the natural analogue of this propagation with respect to an arbitrary amenable subgroup.

Note that the IID Poisson weakly factors onto the Poisson of any intensity.

\begin{defn}
Let $F \subseteq A$ be a finite volume subset. The \emph{Cox process driven by $F$} has driving measure
\[
	\eta_F = \sum_{x \in \Pi_F} \lambda_{xA}(\bullet \cap xF),
\]
where $\Pi_F$ denotes the Poisson point process on $G$ of intensity $\lambda_A(F)^{-1}$. 
\end{defn}

Note that the Cox process driven by $F$ always has intensity one.

By choosing an appropriate sequence of Folner sets $F_n \subseteq A$, we will show that the Cox process driven by $F_n$ weakly converges to the Cox process driven by $G/A$. 

\begin{proof}[Proof of Theorem \ref{constructweakfactor}]
	Choose a sequence $B_i \subseteq G$ of compact stochastic continuity sets which determine weak convergence. Choose  inductively symmetric Folner sets $F_n$ which are $((B_i^{-1}B_i \cap A),\e_n)$-invariant for each $i \leq n$, where $\e_n$ tends to zero. 
	
	Observe that the Cox process driven by $F_n$ is a factor of the IID marked Poisson on $G$. It therefore suffices to show that they weakly converge to the Cox process driven by $G/A$. We prove this using the condition of Theorem \ref{convergencecondition}.

Denote the driving measure by $\eta_n$, and by $\eta$ the driving measure of the Cox process with respect to $G/A$. Let $B$ be one of the compact stochastic continuity sets above. We must check that $\eta_n(B)$ weakly converges to $\eta(B)$. Observe that
\[
	\eta_n(B) = \sum_{x \in \Pi_{F_n}} \lambda_{xA}(B \cap xF_n) = \sum_{x \in BF_n \cap \Pi_{F_n}} \lambda_{xA}(B \cap xF_n),
\]
as terms from $x$ outside of $BF_n$ contribute nothing to the sum. This is a sum of Poisson-many IID random variables, which is known as a compound Poisson distribution. The same is true for our alleged limit $\eta(B)$, as
\[
	\eta(B) = \sum_{gA \in \Upsilon} \lambda_{gA}(B) = \sum_{gA \in \Upsilon \cap BA} \lambda_{gA}(B).
\]
The Poisson parameters of $\eta_n(B)$ and $\eta(B)$ respectively are 
\[
	p_n = \frac{\lambda_G(BF_n)}{\lambda_A{F_n}} \text{ and } p = \lambda_Q(BA),
\]
and the corresponding IID random variables are
\[
	X_n = \lambda_{x_n A}(B \cap x_n F_n) \text{ and } X = \lambda_{gA}(B)
\]
where $x_n \in G$ ranges $\lambda_G$ uniformly over $BF_n$ and $gA \in G/A$ ranges $\lambda_Q$ uniformly over $BA$.

Weak convergence can be understood in terms of the characteristic functions of these random variables. These have explicit forms (see Example 3.8.4 of Durrett, for instance). We are thus tasked to show that $p_n$ converges to $p$ and $X_n$ converges weakly to $X$. 

For the first, observe that by the equation relating Haar measures (Equation \ref{haarrel})
\[
	p_n = \frac{1}{\lambda_A(F_n)} \int_{BA} \lambda_{gA}(BF_n) d\lambda_Q(gA).
\]

The integrand vanishes unless $gA$ is in $BA$. In that case, write $g = ba$ and observe that
\[
	\lambda_{gA}(BF_n) = \lambda_A(a^{-1}b^{-1}BF_n) = \lambda_A(b^{-1}BF_n) \geq \lambda_A(F_n),
\] 
as $b^{-1}BF_n \supseteq F_n$. Thus we have the inequality $p_n \geq \lambda_Q(BA)$. 

Observe that $BF \cap A = (B \cap A)F$, as $F$ is a subset of the subgroup $A$. Hence also

\[
	p_n = \frac{1}{\lambda_A(F_n)} \int_{BA} \lambda_A(b^{-1}BF_n) d\lambda_Q(gA) \leq (1+\e_n) \int_{BA}d\lambda_Q(gA) = (1+\e_n)\lambda_Q(BA),
\]
showing that $p_n$ converges to $p$.

We now show that $X_n$ weakly converges to $X$. First, let $x_n$ denote a uniformly chosen element of $BF_n$, and $bA$ a uniformly chosen element of $BA$. Then $x_n A$ weakly converges to $bA$. To see this, write
\[
\PP[x_n A \in CA] = \frac{\lambda_G(BF_n \cap CA)}{\lambda_G(BF)} = \left(\frac{\lambda_G(BF_n \cap CA)}{\lambda_G(F_n)}\right) \bigg/ \left(\frac{\lambda_G(BF_n)}{\lambda_G(F_n)}\right) \to \frac{\lambda_Q(CA)}{\lambda_Q(BA)},
\]
which is exactly $ \PP[bA \in CA]$, as desired. The convergence of the denominator was the $p_n$ calculation, and the convergence of the numerator works the same. Note that the map $G/A \to \RR_{\geq 0}$ given by $gA \mapsto \lambda_{gA}(B)$ is continuous, hence $\lambda_{x_n A}(B)$ weakly converges to $\lambda_{bA}(B)$. We next claim that $\lambda_{x_n A}(B \cap x_n F_n)$ and $\lambda_{x_n A}(B)$ have the same weak limit. In fact, with probability tending to one we have
\[
	\lambda_{x_n A}(B \cap x_n F_n) = \lambda_{x_n A}(B).
\]
To see this, recall that $F_n$ was chosen to be $((B_i^{-1}B_i \cap A),\e_n)$-invariant. Thus for most points $f \in F_n$ (those in the interior), we have
\[
	F_n \supseteq (B_i^{-1}B_i \cap A)f = B_i^{-1}B_i f \cap Af = B_i^{-1}B_i f \cap A.
\]
Note that points of the form $B_i f$ with $f$ in the interior represent most points of $B_i F_n$. Thus most points $x$ of $B_i F_n$ satisfy
\[
	F_n \supseteq B_i^{-1} x \cap A.
\]
We interpret this statement as follows: if one can write $a = b^{-1} x$ with $b \in B_i$, then $a$ is in $F_n$. For such $x$ we have $B \cap xA = B \cap xF_n$. One inclusion is trivial, and the other follows from the above: an element of the lefthand side can be written $b = xa$, thus $a = b^{-1}x$ is in $F_n$.

Recall the coupling inequality of total variation: if $(X, Y)$ are coupled random variables with distributions $\mu$ and $\nu$, then
\[
	d_{TV}(\mu, \nu) \leq \PP[X \neq Y].
\]
Thus we have shown that
\[
	\lim_{n \to \infty} d_{TV}(\lambda_{x_n A}(B \cap x_n F_n), \lambda_{x_n A}(B)) = 0,
\]
which implies $\lambda_{x_n A}(B)$ and $\lambda_{x_n A}(B \cap x_n F_n)$ have the same weak limit, which we already identified as $\lambda_{bA}(B)$. This concludes the proof.
\end{proof}
\section{Cost one for the Cox process}\label{costonesection}

\subsection{Proof outline}

We first outline the strategy of proof for Theorem \ref{costtheorem}

As discussed above, it suffices to show that the Cox process driven by $G/A$ has fixed price one. Thus we must produce connected factor of IID factor graphs of this Cox process with average degree at most $2 + \e$ for every $\e > 0$. We construct large ``hyperfinite'' factor graphs which have average degree two, and very sparse families of large ``stars'' which have arbitrarily small average degree. The union of these two graphs is connected, which finishes the proof.

\begin{remark}
The star graphs in the proof also appear in \cite{FMW}. In essence, what is shown there is that if one has a Poisson point process equivariantly marked by a doubly recurrent invariant random tessellation, then the analogous hyperconnectivity property holds. 
\end{remark}

\begin{prop}[Hyperfiniteness]\label{hyperfiniteness}
    Let $A < G$ be a noncompact, amenable, unimodular subgroup. Then the Cox process $\Pi$ on $G$ with respect to $G/A$ admits a factor graph $\mathscr{H}(\Pi)$ whose connected components span $gA \cap \Pi$ for $gA \in \Upsilon$ and are isomorphic to $\ZZ$. 
\end{prop}

\begin{defn}
    Let $H < G$ be a closed subgroup. We say that two cosets $g_1 H, g_2 H \in G/H$ are \emph{highly $R$-adjacent} if there is infinite and unbounded set of pairs of points $x \in g_1 H$ and $y \in g_2 H$ with $d(x, y) < R$.
\end{defn}

    We stress that the set should be unbounded, not simply noncompact. 

\begin{prop}[Hyperconnectivity]\label{hyperconnectivity}
	Denote by $\Upsilon$ the Poisson point process on $G/A$. Then almost surely for every distinct pair of cosets $g_1 A$ and $g_2 A$, there exists $R > 0$ such $g_1 A$ and $g_2 A$ are highly $R$-adjacent.
\end{prop}

Note that the $R$ will depend in general on the particular pair of cosets.

These two propositions quickly yield:

\begin{proof}[Proof of Theorem \ref{costtheorem}]

Let $\e > 0$. We construct a connected factor of IID factor graph of average degree less than $2 + \e$.

Choose $t_n$ going to zero sufficiently fast so that 
\[
\sum_n t_n \lambda(B(0,n)) < \e.
\]
Then
\[
\bigstar([0,1]^\Pi) = \bigcup_n \bigstar_{n, t_n}([0,1]^\Pi)
\]
has average degree less than $\e$. Then $\bigstar([0,1]^\Pi) \cup \mathscr{H}(\Pi)$, where $\mathscr{H}(\Pi)$ denotes the hyperfinite graphing from Proposition \ref{hyperfiniteness} is connected almost surely. To see this, note that every individual coset is connected. It now suffices to find even a single edge between an arbitrary pair of cosets $g_1 A$ and $g_2 A$ (we will find infinitely many). Now given two cosets $g_1 A$ and $g_2 A$ are in $\Upsilon$, they are highly $R$ adjacent for some $R > 0$. Moreover, the points of $\Pi$ appear in these cosets as independent Poisson point processes (with intensity measure $\lambda_{g_1 A}$ and $\lambda_{g_2 A}$ respectively). We therefore find infinitely many pairs of points $x \in \Pi \cap g_1 A$ and $y \in \Pi \cap g_2 A$ with $d(x, y) < R$, and hence by Borel-Cantelli infinitely many edges of star graphs connecting such pairs, concluding the proof.
\end{proof}

\subsection{Proof of hyperfiniteness}

\begin{proof}[Proof of Proposition \ref{hyperfiniteness}]
We have seen that the Palm equivalence relation of $\Pi$ with the restricted rerooting equivalece relation $\Rel\Rel$ is $\mu_0$ amenable. It is also aperiodic, as $A$ is noncompact. By Ornstein-Weiss there exists a pmp action $\ZZ \acts (\MMo(G), \mu_0)$ generating $\Rel\Rel$. The induced graphing does the job.
\end{proof}
\subsection{Proof of hyperconnectivity}

To prove Proposition \ref{hyperconnectivity}, we first prove an auxiliary lemma.

\begin{lem}\label{asRconnected}
    Let $H < G$ be a closed and unimodular subgroup such that the action $G \acts G/H \times G/H$ is conservative. Denote by $\Upsilon$ the Poisson point process on $G/H$, which we also view as a random closed subset of $G$. Then almost surely, for every $g_1 H, g_2 H \in \Upsilon$, there exists $R > 0$ such that $g_1 H$ and $g_2 H$ are highly $R$-adjacent.
\end{lem}

\begin{proof}[Proof of Proposition \ref{hyperconnectivity}]
We show that the set of pairs of cosets $(g_1 H, g_2 H)$ for which there exists an $R$ such that they are highly $R$-adjacent has full measure as a subset of $G/H \times G/H$. We are then finished by Corollary \ref{Meckecor}.

Consider the set of cosets which are ``$2r$-adjacent at the origin''. Formally, let
\[
\mathcal{A}_r = \{(g_1 H, g_2 H) \in G/H \times G/H \mid B(0, r) \cap g_1 H \neq \empt \text{ and } B(0, r) \cap g_2 H \neq \empt \}.
\]
Note that if $(g_1 H, g_2 H)$ is in $\mathcal{A}_r$ then there exist $x \in g_1 H$ and $y \in g_2 H$ with $d(x, y) < 2r$.

Observe that if $\gamma \in G$ is such that $\gamma g_1 H \in \mathcal{A}_r$ and $\gamma g_2 H \in \mathcal{A}_r$, then there exists $x \in g_1 H$ and $y \in g_2 H$ such that
\[
d(x, y) = d(\gamma x, \gamma y) < r.
\]
That is, for a given $g_1 H$ and $g_2 H$, if
\[
\{ \gamma \in G \mid (\gamma g_1 H, \gamma g_2 H) \in \mathcal{A}_r \text{ is unbounded}\}
\]
then $g_1 H$ and $g_2 H$ are highly $r$-adjacent. Moreover, $\bigcup_r \mathcal{A}_r = G/H \times G/H$. 

Let $\mathcal{A}'_r$ denote the recurrent points of $\mathcal{A}_r$. Then by conservativity, almost every point of $\mathcal{A}_r$ is in $\mathcal{A}'_r$. We conclude that $\bigcup_n \mathcal{A}_n$ has full measure, as desired. 
\end{proof}

\begin{remark}
In the special case where the subgroup $A$ is normal, one can do a simpler construction in lieu of the hyperconnectivity argument. Namely, one simply lifts a graphing from $G/A$ to a very sparse subset of the Cox process.
\end{remark}


\section{Applications}\label{applicationssection}

\begin{defn}
    We say that a noncompact closed subgroup $H < G$ is \emph{weakly normal} if $H \cap H^g$ is noncompact for almost every $g \in G$.
\end{defn}

Observe that if $A < G$ is an amenable, closed, unimodular, weakly normal subgroup then $G \acts (G/A)^2$ is a conservative action and therefore has fixed price one by Theorem \ref{costtheorem}. This is how we will apply the theorem in practice.

The proof makes use of structure theory for real semisimple Lie groups and $p$-adic split reductive groups in a similar fashion as Lemma 4.1 of \cite{FMW}. Note that the structure theory is quite similar in both cases. We recall this terminology (see Section 4.1 of \cite{FMW} for further explanation).

Let $G_1$ be a semi-simple real Lie group and $G_2$ a connected $p$-adic split reductive group of rank at least one. Take maximal compact subgroups $K_i$, a maximal split tori $A_i$ and minimal parabolics $P_i$ containing $A_i$ in $G_i$. Let $P_1 = M_1 A_1 N_1$ be the Langlands decomposition of $P_1$ (see (2.3.6) of \cite{GangolliVaradarajan}), where $N_1$ is the unipotent radical and $M_1$ is the maximal compact subgroup of the centralizer of $A_1$ in $G_1$. Let $P_2 = M_2 N_2$ be the Levi decomposition (see Section 0.1 of \cite{Silberger}). With $\mathfrak{a}$ the Lie algebra of $A_1$, we have the following formula for the modular character $\chi_{P_1}$:
\[
    \chi_{P_1}(e^H mn) = e^{2\rho_1(H)} \text{ for any } H \in \mathfrak{a}, m \in M, \text{ and } n \in N.
\]
For the $p$-adic case we also have (see Lemma 1.2.1.1 of \cite{Silberger}) the following formula for the modular character $\chi_{P_2}$:
\[
\chi_{P_2}(mn) = \abs{2\rho_2(m)}_p,
\]
where we have simplified the notation. Here $2\rho(m)$ is the sum of the roots, see Section 1.2.1 of \cite{Silberger} for further information.

\subsection{Products with real Lie groups}

\begin{proof}[Proof of Theorem \ref{higherrankfp}]
 Observe that in higher rank $G_1$ contains an amenable, closed, unimodular, weakly normal subgroup $A$ by Lemma 4.1 of \cite{FMW}. Then $A \times 1 < G_1 \times G_2$ also satisfies these conditions, giving fixed price one.

For the second case, by a similar argument it suffices to find an amenable, closed, unimodular, weakly normal subgroup of $G_1 \times G_2$. We execute a similar argument to Lemma 4.1 of \cite{FMW}. 

Let $U$ denote the kernel of the modular function of $P_1 \times P_2$. Then $U$ is amenable, closed, and unimodular. We show that $U \cap U^{(g_1,g_2)}$ is noncompact for almost every $(g_1,g_2) \in G_1 \times G_2$.

We make use of the Bruhat decompositions of $G_i$ (see (2.5.5) of \cite{GangolliVaradarajan} for the real case and Section 19 of \cite{Milne} for the $p$-adic). We may write
\[
G_i = \bigsqcup_{w_i \in W_i} P_i w_i P_i,
\]
where the decomposition is over the respective Weyl groups. In both decompositions there is a ``big Bruhat cell'' $P_i w^i_0 P_i$ corresponding to the longest element of the Weyl group, which is characterised by the fact that $\rho_i^{w^i_0} = -\rho_i$. Now fix $(g_1, g_2) \in G_1 \times G_2$. Ignoring a null set, we may write $g_i = p_i^1 w_0^i p_i^2$ with $p_i^1$ and $p_i^2$ in $P_i$. Then
\[
U \cap U^{(g_1, g_2)} = U \cap U^{(p_1^1 w_0^1 p_1^2,p_2^1 w_0^2 p_2^2)} = (U \cap U^{(w_0^1, w_0^2)})^{(p_1^1,p_2^1)},
\]
where we make use of normality. We show that $U \cap U^{(w_0^1, w_0^2)}$ contains the following closed subgroup, and then show it is noncompact.
\[
D := \ker(\chi_{P_1}\chi_{P_2}) =  \{ (g_1, g_2) \in P_1 \times P_2 \mid \chi_{P_1}(g_1) \chi_{P_2}(g_2) = 1\}.
\]
Clearly $D \subseteq U$. We show $D$ is contained in $U^{(w_0^1, w_0^2)}$. Take $(g_1, g_2) \in D$ and write $g_1 = m_1 a n_1$ and $g_2 = m_2 n_2$. Then by definition
\[
\chi_{P_1}(g_1^{w_0^1}) = \frac{1}{\chi_{P_1}(g_1)} \text{ and } \chi_{P_2}(g_2^{w_0^2}) = \frac{1}{\chi_{P_2}(g_2)},
\]
and thus $(g_1, g_2)$ is in $U^{(w_0^1, w_0^2)}$, as desired.

Finally, for noncompactness of $D$ note that $\chi_{P_1}\chi_{P_2}$, modulo the Levi subgroups of each factor, descends to a linear form on $\RR^{\text{rank}(G_1)} \times \QQ_p^{\text{rank}(G_2)}$, and thus has noncompact kernel.
\end{proof}

\subsection{$\SL(2,\QQ)$}
\begin{proof}[Proof of Theorem \ref{sl2qfp}]

Gaboriau showed\cite{GaboriauNotes} that if a countable group $\Gamma$ is the increasing union of fixed price one groups, then $\Gamma$ has fixed price one. Observe that $\SL(2,\QQ) = \bigcup_{n}\SL(2, \ZZ[\frac{1}{p_1}, \frac{1}{p_2}, \ldots, \frac{1}{p_n}])$,
and each $\SL(2, \ZZ[\frac{1}{p_1}, \frac{1}{p_2}, \ldots, \frac{1}{p_n}]) < \SL(2, \RR) \times \SL(2, \QQ_{p_1}) \times \cdots \times \SL(2, \QQ_{p_n})$ is a lattice. These groups have fixed price one by Theorem \ref{higherrankfp}, proving the result.
\end{proof}

\printbibliography

\end{document}